\documentclass[12pt,reqno]{amsart}

\usepackage{graphicx}
\graphicspath{ {./images/} }
\usepackage{bm}

\usepackage{color}

\usepackage{amssymb}

\usepackage{amsfonts}

\usepackage{amsmath}

\usepackage[all]{xy}

\usepackage[colorlinks=true, allcolors=blue]{hyperref}

\newtheorem{theorem}{Theorem}[section]

\newtheorem{corollary}[theorem]{Corollary}

\newtheorem{lemma}[theorem]{Lemma}

\newtheorem{proposition}[theorem]{Proposition}

\newtheorem{conjecture}[theorem]{Conjecture}

\newtheorem{Definition}[theorem]{Definition}

\newtheorem{Example}[theorem]{Example}

\newtheorem{Remark}[theorem]{Remark}

\newcommand{\R}{\mathbb{R}}
\newcommand{\Rp}[1]{\R_{+}^{#1}}
\newcommand{\Rpp}[1]{\R_{++}^{#1}}
\newcommand{\M}[2]{M_{#1,#2}}
\newcommand{\one}{\mathbf{1}}
\newcommand{\ip}[2]{\left\langle #1,#2\right\rangle}
\newcommand{\Span}{\operatorname{span}}

\address{Vatsalkumar N. Mer \\ Institute for Industrial and Applied Mathematics, Chungbuk National University, Cheongju 28644, Korea}
\email{vnm232657@gmail.com}

\begin{document}

\title[Linear preserver problems]{Invertible linear preservers of semipositive matrices - a dimension free approach}
\author[Vatsalkumar]{Vatsalkumar N. Mer }


\begin{abstract}
An \(m\times n\) real matrix \(A\) is said to be semipositive if there
exists a vector \(x>0\) such that \(Ax>0\), where the inequalities are
understood componentwise. Dorsey et al. \cite{DGJJT} conjectured that any invertible
linear map that $L$ that leaves invariant the collection of all semipositive matrices is always in the standard form
$A \mapsto XAY$ for some row positive matrix $X$ and inverse nonnegative matrix $Y$. This was settled in
\cite{JM} when $m \geq n$. Our aim in this paper is to settle the case when \(m<n\) of the above conjecture. In fact, our proof works for arbitrary positive integers \(m\) and \(n\). The main ingredient is a classification of affine subspaces of the largest possible dimension contained in the set of semipositive matrices.

\vspace{5mm}


\end{abstract}

\maketitle

\section{Introduction}
 Let \(M_{m,n}(\mathbb R)\) denote the set of all \(m\times n\) matrices over \(\R\). When \(m=n\), this set will 
 also be denoted by \(M_n (\mathbb{R})\).  For \(n\geq1\), let
$ \Rp{n}=\{x\in\R^n:x_i\geq0\ \text{for every }i\},
 \ \Rpp{n}=\{x\in\R^n:x_i>0\ \text{for every }i\}.
$A matrix \(A\in M_{m,n}(\mathbb R)\) is said to be \emph{semipositive} if there exists a
vector \(x \in \Rp{n}\) such that \(Ax \in \Rpp{m}\). We denote the set of all semipositive matrices 
by $S_{m,n}=S(\R_+^n,\R_+^m).$ Semipositive matrices were studied systematically in
\cite{CSM, prs-1}; they also occur naturally in the theory of nonnegative
matrices and \(M\)-matrices (see, for instance,
\cite[Chapter~6]{BermanPlemmons}).  For further geometric properties
of semipositive matrices, one may refer to \cite{Tsatsomeros} and the
references cited therein.

For a field \(\mathbb F\), a linear preserver on
\(M_{m,n}(\mathbb F)\) is a linear map that preserves a specified
property or relation.  Given a subset \(\mathcal S\) of
\(M_{m,n}(\mathbb F)\), one may ask for all linear maps \(L\) such
that either \(L(\mathcal S)\subseteq\mathcal S\) or
\(L(\mathcal S)=\mathcal S\).  The former maps are called \emph{into
preservers}, while the latter are called \emph{onto} or \emph{strong
preservers}.  Linear preserver problems have a long history in matrix
theory; see \cite{LiPierce} for an overview.  The present manuscript
concerns invertible into preservers of \(S_{m,n}\).

Recall that a square matrix is \emph{row positive} if it is
entrywise nonnegative and every row contains a nonzero entry.  An
invertible square matrix \(Y\) is said to be \emph{inverse
nonnegative} if \(Y^{-1}\geq0\).  In \cite{DGJJT}, Dorsey et al.\ proved that, for
fixed square matrices \(X\) and \(Y\), the map
\(A\mapsto XAY\) is an into preserver of semipositivity precisely when
\(X\) is row positive and \(Y\) is inverse nonnegative, or when the
same conditions hold for \(-X\) and \(-Y\)
\cite[Theorem $2.4$]{DGJJT}.  This result led them to the following
conjecture.

\begin{conjecture}[Dorsey et al.\ \cite{DGJJT}]
\label{conj:dorsey}
Let \(L: M_{m,n}(\mathbb R)\to M_{m,n}(\mathbb R)\) be an invertible linear map.  If
$
       L(S_{m,n})\subseteq S_{m,n},
$
then \(L(A)=XAY\) for every \(A\in M_{m,n}(\mathbb R)\), where \(X\) is row
positive and \(Y\) is inverse nonnegative.
\end{conjecture}

Dorsey et al.\ also showed by an example that the invertibility of
\(L\) cannot be omitted and reported computational evidence for the
\(2\times2\) case.  Jayaraman and Mer \cite{JM} subsequently obtained two
results that are particularly relevant here.  First, they proved that
the conclusion of Conjecture \ref{conj:dorsey} follows if every
rank-one semipositive matrix is mapped to a rank-one matrix
\cite[Theorem $3.1$]{JM}.  They then established the conjecture for
\(m\geq n\geq2\) \cite[Theorem $3.11$]{JM}.  Although their sufficient
rank-one result does not require a relation between \(m\) and \(n\),
the argument used to obtain the required rank-one images was carried
out in the range \(m\geq n\).  A recent survey records precisely this
dimensional range \cite{CY}. Thus the rectangular case \(m < n\) was
not covered by that method.

The purpose of this manuscript is to settle this remaining case.  Our
argument is independent of \(m\) and \(n\), and
therefore gives the following dimension-free form of the conjecture.

\begin{theorem}[Main theorem]
\label{thm:main}
Let \(m,n\geq2\), and let \(L:M_{m,n}(\mathbb R)\to M_{m,n}(\mathbb R)\) be an invertible
linear map.  If
$
       L(S_{m,n})\subseteq S_{m,n},
$
then there exist invertible matrices \(X\in M_m (\mathbb R)\) and \(Y\in M_n (\mathbb R)\)
such that
$
       L(A)=XAY\ (A\in M_{m,n}(\mathbb R)),
$
where \(X\) is row positive and \(Y\) is inverse nonnegative.
Conversely, every map of this form is an into preserver of
\(S_{m,n}\).
\end{theorem}

The proof differs from the rank-one approach used previously.  For
\(0\ne x\geq0\) and \(b>0\), consider the affine fiber
\[
\mathcal P_{x,b}=\{A\in M_{m,n}(\mathbb R):Ax=b\}.
\] 
We prove that \(mn-m\) is the largest possible dimension of an affine
subspace contained in \(S_{m,n}\), and that the affine subspaces of
this dimension are exactly the fibers \(\mathcal P_{x,b}\).  An
invertible into preserver must therefore map such fibers to fibers of
the same kind.  After passing to directions and then to Frobenius
orthogonal complements, the adjoint map \((L^{-1})^*\) preserves the
subspaces
\[
       \mathcal U_x=\{ux^T:u\in\R^m\},
\]
which form one ruling of the rank-one tensor variety.  An elementary
factorization argument then gives \(L(A)=XAY\).  

We begin in Section $2$ with a theorem of the alternative and an affine
form of Farkas' lemma.  In Section $3$, we classify the
maximal-dimensional affine subspaces contained in the set of semipositive matrices.  
In Section $4$, we use this classification to obtain the product form of
an invertible into preserver.  Finally, in Section $5$, we present the proof of  the main theorem (Theorem~\ref{thm:main}).

\section{Preliminary results}

We present the preliminary results needed in the proof. 

\begin{lemma}[\cite{JKS}, Lemma  $2.1$]
\label{lem:witness}
For \(A\in\M{m}{n}\), the following statements are equivalent:
\begin{enumerate}
\item there exists \(x\in\Rpp{n}\) such that \(Ax\in\Rpp{m}\);
\item there exists \(0\ne x\in\Rp{n}\) such that \(Ax\in\Rpp{m}\).
\end{enumerate}
\end{lemma}

We shall also have an occasion to use the following Theorem of the Alternative. The proof of this can be found in \cite{chs}.

\begin{theorem}[Theorem of the Alternative]
\label{thm:alternative}
For \(A\in\M{m}{n}\), exactly one of the following statements hold:
\begin{enumerate}
\item there exists \(0\ne x\in\Rp{n}\) such that \(Ax\in\Rpp{m}\);
\item there exists \(0\ne y\in\Rp{m}\) such that \(A^Ty\leq0\).
\end{enumerate}
Consequently,
\[
 A\notin S_{m,n}
 \quad\Longleftrightarrow\quad
 \text{there exists }0\ne y\in\Rp{m}\text{ with }A^Ty\leq0.
\]
\end{theorem}

The following lemma is an immediate consequence of the inequality form of
Farkas' lemma; see, for example, \cite[Lemma $10.5$]{Vanderbei}.

\begin{lemma}
\label{lem:affine-farkas}
Let \(a\in\mathbb{R}^d\) and let \(U\) be a linear subspace of
\(\mathbb{R}^d\). Then the following statements are equivalent:
\begin{enumerate}
\item
$(a+U)\cap(-\mathbb{R}_+^d)=\varnothing;$
\item there exists $x\in\mathbb{R}_+^d\cap U^\perp$ such that \(a^Tx>0\).
\end{enumerate}
\end{lemma}

\begin{proof}
If \(r=\dim U\), choose \(B\in\mathbb{R}^{d\times r}\) such that
\(\operatorname{range}(B)=U\). Then
\[
(a+U)\cap(-\mathbb{R}_+^d)\neq \varnothing
\]
if and only if there exists \(z\in\mathbb{R}^r\) satisfying
\[
Bz\leq -a.
\]
Writing \(z=z^+-z^-\), with \(z^+,z^- \geq0\), and introducing a slack
variable \(s\geq0\), the latter system is equivalent to
\[
\begin{bmatrix}B&-B&I_d\end{bmatrix}
\begin{bmatrix}z^+\\ z^-\\ s\end{bmatrix}
=-a,
\qquad
\begin{bmatrix}z^+\\ z^-\\ s\end{bmatrix}\geq0.
\]
By Farkas' lemma, this system is infeasible if and only if there exists
\(x\in\mathbb{R}^d\) such that
\[
B^Tx\geq0,\qquad -B^Tx\geq0,\qquad x\geq0,
\qquad (-a)^Tx < 0.
\]
The first two inequalities imply \(B^Tx = 0\), while the last inequality is
equivalent to \(a^Tx > 0\). Since
\[
\ker(B^T)=\operatorname{range}(B)^\perp=U^\perp,
\]
the conclusion follows.
\end{proof}

\section{Maximal affine subspaces in the semipositive set}

We now turn our attention to affine subspaces contained in
\(S_{m,n}\).  The following theorem is the main geometric ingredient
in the proof of Theorem \ref{thm:main}.

For \(0\ne x\in\R^n\), define
\[
\mathcal W_x=\{H\in\M{m}{n}:Hx=0\},
\qquad
\mathcal U_x=\{u x^T:u\in\R^m\}.
\] 
The evaluation map \(H \mapsto Hx\) being surjective, we have 
\[
\dim\mathcal W_x = mn-m.
\]
Moreover, 
\begin{equation}
\label{eq:orthogonal-rulings}
 \mathcal W_x^\perp = \mathcal U_x,
\end{equation} 
as 
\(\ip {H}{ux^T} = u^THx\).
 The direction space of a nonempty affine subspace
\(\mathcal A\) is
\[
    \operatorname{dir}(\mathcal A)
    :=\{A_1-A_2:A_1,A_2\in\mathcal A\}.
\]
\begin{theorem}
\label{thm:affine-flats}
Let  \(A_0\in S_{m,n}\) and \(\mathcal P=A_0+\mathcal W\) be an affine subspace of
\(\M{m}{n} (\mathbb{R} )\) such that \(\mathcal P\subseteq S_{m,n}\).  Then
\[
\dim\mathcal P\leq mn-m.
\] 
Equality holds if and only if there exist
\(0\ne x\in\Rp{n}\) and \(b\in\Rpp{m}\) such that
\begin{equation}
\label{eq:fiber}
\mathcal P=\mathcal P_{x,b}
:=\{A\in\M{m}{n}:Ax=b\}.
\end{equation}
\end{theorem}

\begin{proof}
Fix \(0\ne y\in\Rp{m}\), and set
\[
a_y=A_0^Ty,
\qquad
U_y=\{H^Ty:H\in\mathcal W\}\subseteq\R^n.
\]
We claim that $(a_y+U_y)\cap(-\Rp{n})=\varnothing.$
Indeed, if \(A_0^Ty+H^Ty\leq0\) for some \(H\in\mathcal W\), then
\((A_0+H)^Ty\leq0\).  Theorem \ref{thm:alternative} would then imply
\(A_0+H\notin S_{m,n}\), contradicting 
\(\mathcal P \subseteq S_{m,n}\).

By Lemma \ref{lem:affine-farkas}, there exists a vector
\(0\ne x_y\in\Rp{n}\) such that
\begin{equation}
\label{eq:xy-properties}
        x_y\in U_y^\perp,
       \
       y^TA_0x_y=a_y^Tx_y>0.
\end{equation}

For every \(H\in\mathcal W\), the first relation gives
$ \ip{H}{yx_y^T}=y^THx_y=0.$ Therefore 

\begin{equation}
\label{eq:rank-one-normal}
yx_y^T\in\mathcal W^\perp.
\end{equation}

Applying the above with \(y=e_i\), \(1\leq i\leq m\), we obtain
nonzero vectors \(x_i\in\Rp{n}\) for which
\[
e_ix_i^T\in\mathcal W^\perp.
\]
These \(m\) matrices are linearly independent, because each is
supported in a different row and is nonzero.  Therefore, 
$ \dim\mathcal W^\perp\geq m,$
and consequently $ \dim\mathcal P=\dim\mathcal W\leq mn-m.$

Assume now that equality holds.  Then
\(\dim\mathcal W^\perp=m\), thereby making
$ \{e_1x_1^T,\dots,e_mx_m^T\}$ a basis of \(\mathcal W^\perp\).  Apply
\eqref{eq:rank-one-normal} with \(y=\one = (1,1, ...,1)^T \in\Rpp{m}\).  Then, there is a
nonzero \(x_*\in\Rp{n}\) such that
\(\one x_*^T\in\mathcal W^\perp\).  Hence, for suitable scalars
\(c_1,\dots,c_m\),
\[
       \one x_*^T=\sum_{i=1}^m c_i e_ix_i^T.
\]
Comparing the \(i\)-th rows gives
$
       x_*^T=c_i x_i^T, \ (1\leq i\leq m).
$
Both \(x_*\) and \(x_i\) being nonzero and nonnegative, each
\(c_i\) is strictly positive.  Thus every \(x_i\) is a positive scalar
multiple of \(x_*\), and
\[
 \mathcal W^\perp
   =\Span\{e_1x_*^T,\dots,e_mx_*^T\}
   =\mathcal U_{x_*}.
\]
By \eqref{eq:orthogonal-rulings},
\(\mathcal W=\mathcal W_{x_*}\).  Furthermore,
\eqref{eq:xy-properties} with \(y=e_i\) says
\(e_i^TA_0x_i>0\).  Since \(x_i\) is a positive scalar multiple of
\(x_*\), we obtain
\[
       b:=A_0x_*\in\Rpp{m}.
\]
It then follows that
\[
 A_0+\mathcal W
 =A_0+\mathcal W_{x_*}
 =\{A:Ax_*=A_0x_*\}
 =\mathcal P_{x_*,b},
\]
which proves the necessity in the equality case.

Conversely, let \(0\ne x\in\Rp{n}\) and \(b>0\).  Every
\(A\in\mathcal P_{x,b}\) satisfies \(Ax=b>0\). 
Lemma \ref{lem:witness} then gives \(A\in S_{m,n}\).  The direction of this
fiber is \(\mathcal W_x\), whose dimension is \(mn-m\).  This proves
both the converse and the final assertion.
\end{proof}

\section{The structure of an invertible into preserver}

We now use the classification of affine-subspaces to determine the
algebraic form of an invertible into preserver.  We begin with the following factorization lemma.

\begin{lemma}
\label{lem:ruling-factorization}
Let \(T:M_{m,n}(\mathbb R)\to M_{m,n}(\mathbb R)\) be an invertible linear map. Suppose
that for every \(x\in\mathbb R_{++}^n\), there exists a nonzero
\(\phi(x)\in\mathbb R^n\) such that
\begin{equation}
\label{eq:ruling-hypothesis}
T(\mathcal U_x)=\mathcal U_{\phi(x)},
\end{equation}
where
$
\mathcal U_x:=\{ux^T:u\in\mathbb R^m\}.
$
Then there exist invertible matrices  \(P\in M_m(\mathbb R)\) and
\(Q\in M_n(\mathbb R)\) such that
$
T(A)=PAQ.
$
\end{lemma}

\begin{proof}
The proof is divided into two claims. We state and prove them in order and deduce the required 
result  from them.

\noindent
\underline{Claim $1$:}
For every $x\in\mathbb R_{++}^n\cup(-\mathbb R_{++}^n),$ there exists \(0\neq p_x\in\mathbb R^n\) 
such that $T(\mathcal U_x)=\mathcal U_{p_x}.$

\smallskip
\noindent
\underline{Proof of Claim $1$:}
If \(x\in\mathbb R_{++}^n\), then the claim is trivially true. If \(x\in-\mathbb R_{++}^n\), then 
\(-x\in\mathbb R_{++}^n\) and $\mathcal U_x = \mathcal U_{-x}.$ We then obtain
$T(\mathcal U_x) = T(\mathcal U_{-x}) = \mathcal U_{\phi(-x)}$, which proves the claim.

\medskip
\noindent
\underline{Claim $2$:}
For every \(0\neq z\in\mathbb R^n\), there exists
\(0\neq p_z\in\mathbb R^n\) such that $T(\mathcal U_z)=\mathcal U_{p_z}.$

\smallskip
\noindent
\underline{Proof of Claim $2$:}
If $ z \in \mathbb R_{++}^n\cup(-\mathbb R_{++}^n),$ the result follows from Claim $1$. We therefore assume 
that $z \notin \mathbb R_{++}^n \cup (-\mathbb R_{++}^n)$. Let us choose \(h\in\mathbb R_{++}^n\) such that \(h\) 
is not a scalar multiple of \(z\). For sufficiently large \(t>0\), the vectors
$z_1:=z+th \ \text{and}\ z_2:= th$ belong to \(\mathbb R_{++}^n\). Moreover,
$z = z_1-z_2$ and are also linearly independent. Indeed, if
they were linearly dependent, then $z= z_1-z_2$ would be a scalar multiple of \(h\), contrary to the choice of \(h\). 
By the hypothesis, there exist nonzero vectors \(p_1, p_2\in\mathbb R^n\) such that
\[
T(\mathcal U_{z_1})=\mathcal U_{p_1}, \qquad T(\mathcal U_{z_2})=\mathcal U_{p_2}.
\] 
The vectors \(p_1\) and \(p_2\) are again linearly independent. Otherwise,
\(\mathcal U_{p_1} = \mathcal U_{p_2}\), and by the invertibility of \(T\),
$\mathcal U_{z_1}=\mathcal U_{z_2}.$  This would then imply that \(z_1\) and \(z_2\) are scalar multiples of one another, 
a contradiction to their linear independence.

For \(k=1,2\), the restriction of \(T\) to \(\mathcal U_{z_k}\) is an
isomorphism from \(\mathcal U_{z_k}\) onto \(\mathcal U_{p_k}\).
Consequently, there exists a uniquely determined invertible
\(P_k\in M_m(\mathbb R)\) such that
\begin{equation}
\label{eq:Pk-ruling}
 T(uz_k^T)=P_ku\,p_k^T \qquad(u\in\mathbb R^m).
\end{equation}

Since \(z_1+z_2 \in\mathbb R_{++}^n\), Claim $1$ implies that $T\bigl(u(z_1+z_2)^T\bigr)$ has rank at most one 
for every \(u\in\mathbb R^m\). On the other hand, by linearity and \eqref{eq:Pk-ruling},
\[
\begin{aligned}
    T\bigl(u(z_1+z_2)^T\bigr)
    &=T(uz_1^T)+T(uz_2^T)\\
    &=P_1u\,p_1^T+P_2u\,p_2^T.
\end{aligned}
\]
Since \(p_1\) and \(p_2\) are linearly independent, it follows that
\(P_1u\) and \(P_2u\) are linearly dependent for every
\(u\in\mathbb R^m\).

Set $S:= P_1^{-1}P_2.$ Then $Su \in \operatorname{span}\{u\}\ (u\in\mathbb R^m).$
It then follows that $S=\lambda I_m$  for some nonzero scalar \(\lambda\). Thus, $P_2 = \lambda P_1.$ 
Using \(z=z_1-z_2\), we now obtain
\[
\begin{aligned}
    T(uz^T)
    &=T(uz_1^T)-T(uz_2^T)\\
    &=P_1u\,p_1^T-\lambda P_1u\,p_2^T\\
    &=P_1u\,(p_1-\lambda p_2)^T.
\end{aligned}
\] 
Since \(p_1\) and \(p_2\) are linearly independent, $p_1-\lambda p_2 \neq 0.$ From invertibility of \(P_1\), we get 
$T(\mathcal U_z) = \mathcal U_{p_1-\lambda p_2}$, which proves the claim.

\medskip
\noindent
We have thus proved that $L$ preserves the set of rank-one matrices. By the nonsingular rank-one 
preserver theorem \cite[Theorem $2$]{Lautemann}, there  exist invertible matrices  
\(P\in M_m(\mathbb R)\) and \(Q\in M_n(\mathbb R)\) such that either
\begin{equation}
\label{eq:lautemann-standard}
 T(A)=PAQ,
\end{equation}
or \(m=n\) and
\begin{equation}
\label{eq:lautemann-transpose}
T(A)=PA^TQ.
\end{equation}

We now show that the second alternative stated above is not possible. Suppose that
\(m=n\) and that \eqref{eq:lautemann-transpose} holds. Fix
\(x\in\mathbb R_{++}^n\). Then
\[
\begin{aligned}
    T(\mathcal U_x)
    &=\{P(ux^T)^TQ:u\in\mathbb R^n\}\\
    &=\{(Px)(Q^Tu)^T:u\in\mathbb R^n\}\\
    &=\{(Px)v^T:v\in\mathbb R^n\},
\end{aligned}
\]
where the last equality follows from the invertibility of \(Q^T\).
On the other hand, the hypothesis gives
$T(\mathcal U_x)=\mathcal U_{\phi(x)} =\{w\phi(x)^T:w\in\mathbb R^n\}.$ 
Choose $v\notin \operatorname{span}\{\phi(x)\}$, 
which is possible because \(n\geq2\). Then $(Px)v^T\in T(\mathcal U_x).$ 
Thus, there exists \(w\in\mathbb R^n\) such that $(Px)v^T=w\phi(x)^T.$
Both sides being nonzero rank-one matrices, their row spaces
are equal. This then gives  $v \in \operatorname{span}\{\phi(x)\},$
contradicting the choice of \(v\). This proves that the second arternative the transpose cannot occur. 
Therefore, \eqref{eq:lautemann-standard} holds. We finally obtain
\[
T(A)=PAQ \qquad(A\in M_{m,n}(\mathbb R)),
\] as desired.
\end{proof}

\begin{proposition}[Algebraic form of an invertible into preserver]
\label{prop:algebraic-form}
Let \(L:M_{m,n}(\mathbb R)\to M_{m,n}(\mathbb R)\) be invertible and satisfy
\(L(S_{m,n})\subseteq S_{m,n}\).  Then there exist invertible matrices 
\(X\in M_m(\R)\) and \(Y\in M_n(\R)\) such that
\[
L(A)=XAY.
\]
\end{proposition}

\begin{proof}
Fix \(x\in\Rpp{n}\) and \(b=\one\in\Rpp{m}\).  By
Theorem \ref{thm:affine-flats},
\[
\mathcal P_{x,b}=\{A:Ax=b\}
\] 
is an affine subspace of dimension \(mn-m\) contained in \(S_{m,n}\), and
its direction is \(\mathcal W_x\).  Since \(L\) is invertible,
\(L(\mathcal P_{x,b})\) is again an affine subspace of dimension
\(mn-m\) contained in \(S_{m,n}\). Appealing to 
Theorem \ref{thm:affine-flats} again, we get a nonzero
\(z_x\in\Rp{n}\) and \(c_x\in\Rpp{m}\) such that
\[
L(\mathcal P_{x,b})=\mathcal P_{z_x,c_x}.
\] 
Taking directions yields
\begin{equation}
\label{eq:W-preserved}
L(\mathcal W_x)=\mathcal W_{z_x} \qquad(x\in\Rpp{n}).
\end{equation}

Set
\[
T=(L^{-1})^*,
\] where $L^*$ is the adjoint of $L$ with respect to the Frobenius inner product. 
For any subspace \(\mathcal V \subseteq\M{m}{n}\), invertibility of $L$ gives
\begin{equation}
\label{eq:orthogonal-image}
 (L\mathcal V)^\perp = (L^{-1})^*(\mathcal V^\perp).
\end{equation}

In fact, \(B\in(L\mathcal V)^\perp\) if and only if
\(L^*B\in\mathcal V^\perp\), which is equivalent to
\(B\in(L^*)^{-1}(\mathcal V^\perp)=(L^{-1})^*(\mathcal V^\perp)\). 
Using \eqref{eq:orthogonal-rulings}, \eqref{eq:W-preserved}, and
\eqref{eq:orthogonal-image}, we obtain
\[
T(\mathcal U_x)=\mathcal U_{z_x} \qquad(x\in\Rpp{n}).
\]
From Lemma \ref{lem:ruling-factorization}, we get 
invertible matrices  \(P\in M_m(\mathbb R)\) and
\(Q\in M_n(\mathbb R)\)  such that
\[
T(A)=PAQ.
\] 
so that \(T^*(A)=P^TAQ^T\).  Since
\(T^*=L^{-1}\), we conclude that
\[
L(A)=P^{-T}AQ^{-T}.
\] 
Thus the desired representation holds with
\(X=P^{-T}\) and \(Y=Q^{-T}\).
\end{proof}

\section{The main theorem}

We are now in a position to prove the main result stated in the introduction.

\begin{proof}[Proof of Theorem~\ref{thm:main}]
Suppose \(L\) is invertible and \(L(S_{m,n})\subseteq S_{m,n}\).  From Proposition
\ref{prop:algebraic-form} we get invertible matrices \(X\) and \(Y\) such that \(L(A)=XAY\).  Theorem $2.4$ of 
\cite{DGJJT} then shows that \(X\) is row positive and \(Y\) is inverse nonnegative. The converse follows from 
Theorem $2.4$ of \cite{DGJJT}.
\end{proof}

\begin{corollary}[The open rectangular case]
Let \(2\leq m<n\), and let \(L:M_{m,n}(\mathbb R)\to M_{m,n}(\mathbb R)\) be invertible.  If
\(L(S(\Rp{n},\Rp{m}))\subseteq S(\Rp{n},\Rp{m})\), then
\[
L(A)=XAY\qquad(A\in M_{m,n}(\mathbb R)),
\] 
where \(X\in M_m(\R)\) is invertible row positive and\(Y\in M_n(\R)\) is inverse nonnegative.
\end{corollary}

 \section*{Acknowledgments}
 The author would like thank  Sachindranath Jayaraman for his suggestions and comments.
The work of Vatsalkumar N. Mer  was supported by Basic Science Research Program through the National Research Foundation of Korea (NRF) funded by the Ministry of Education, Korea (No. RS-2024-00462498). 

\textbf{Disclosure statement}\\
No potential conflict of interest was reported by the author.

\end{document}